\documentclass[12pt,reqno]{amsart}

\usepackage{amssymb}

\usepackage[all]{xy}

\numberwithin{equation}{section}

\newtheorem{theorem}{Theorem}[section]
\newtheorem{proposition}[theorem]{Proposition}
\newtheorem{lemma}[theorem]{Lemma}
\newtheorem{corollary}[theorem]{Corollary}

\theoremstyle{definition}
\newtheorem{definition}[theorem]{Definition}
\newtheorem{remark}[theorem]{Remark}

\begin{document}

\baselineskip=15pt

\title[Ramified maps and stability of pulled back bundles]{Ramified covering maps and
stability of pulled back bundles}

\author[I. Biswas]{Indranil Biswas}

\address{School of Mathematics, Tata Institute of Fundamental
Research, Homi Bhabha Road, Mumbai 400005, India}

\email{indranil@math.tifr.res.in}

\author[A. J. Parameswaran]{A. J. Parameswaran}

\address{School of Mathematics, Tata Institute of Fundamental
Research, Homi Bhabha Road, Mumbai 400005, India}

\email{param@math.tifr.res.in}



\date{}

\begin{abstract}
Let $f\,:\,C\,\longrightarrow\,D$ be a nonconstant separable morphism
between irreducible smooth projective curves defined over an
algebraically closed field. We say that $f$ is genuinely ramified if
${\mathcal O}_D$ is the maximal semistable subbundle of $f_*{\mathcal O}_C$ (equivalently,
the homomorphism $f_*\, :\, \pi_1^{\rm et}(C)\,\longrightarrow\, \pi_1^{\rm et}(D)$
is surjective). We prove that the pullback
$f^*E\,\longrightarrow\, C$ is stable for every stable vector bundle $E$ on $D$
if and only if $f$ is genuinely ramified.
\end{abstract}

\maketitle

\tableofcontents

\section{Introduction}

Let $k$ be an algebraically closed field. Let $f\,:\,C\,\longrightarrow\,D$ be a nonconstant separable morphism
between irreducible smooth projective curves defined over $k$. For any semistable vector bundle $E$ on $D$, the
pullback $f^*E$ is also semistable. However, $f^*E$ need not be stable for every stable vector bundle
$E$ on $D$. Our aim here is to characterize all $f$ such that $f^*E$ remains stable for every stable vector bundle
$E$ on $D$. It should be mentioned that $E$ is stable (respectively, semistable) if $f^*E$ is
stable (respectively, semistable).

For any $f$ as above the following five conditions are equivalent:
\begin{enumerate}
\item The homomorphism between \'etale fundamental groups
$$f_*\, :\, \pi_1^{\rm et}(C)\,\longrightarrow\, \pi_1^{\rm et}(D)$$
induced by $f$ is surjective.

\item The map $f$ does not factor through some nontrivial \'etale cover of $D$ (in particular, $f$ is not
nontrivial \'etale).

\item The fiber product $C\times_D C$ is connected.

\item $\dim H^0(C,\, f^*f_*{\mathcal O}_C)\,=\,1$.

\item The maximal semistable subbundle of the direct image $f_*{\mathcal O}_C$ is ${\mathcal O}_D$.
\end{enumerate}
The map $f$ is called genuinely ramified if any (hence all) of the above five conditions holds.
Proposition \ref{genuinerampi1surj} and Definition \ref{def1} show
that the above statements (1), (2) and (5) are equivalent; in Lemma
\ref{genuinerampi1surjcurves} it is shown that the statements (3), (4) and (5) are equivalent.

We prove the following (see Theorem \ref{thm2}):

\begin{theorem}\label{thmi}
Let $f\,:\,C\,\longrightarrow\,D$ be a nonconstant separable morphism
between irreducible smooth projective curves defined over $k$. The map $f$ is genuinely ramified 
if and only if $f^*E\, \longrightarrow\, C$ is stable for every stable vector bundle $E$ on $D$.
\end{theorem}

The key technical step in the proof of Theorem \ref{thmi} is the following (see
Proposition \ref{sumofnegetivedegreelinebundles1}):

\begin{proposition}\label{ip}
Let $f\,:\,C\,\longrightarrow\,D$ be a genuinely ramified Galois morphism, of degree $d$,
between irreducible smooth projective curves defined over $k$. Then
$$f^* ((f_* {\mathcal O}_C) / {\mathcal O}_D)\, \subset\, \bigoplus_{i=1}^{d-1} {\mathcal L}_i\, ,$$
where each ${\mathcal L}_i$ is a line bundle on $D$ of negative degree.
\end{proposition}

When $k\,=\, {\mathbb C}$, a vector bundle $F$ on a smooth complex projective curve is
stable if and only if $F$ admits an irreducible flat projective unitary connection
\cite{NS}. From this characterization of stable vector bundles it follows immediately
that given a nonconstant map $f\,:\,C\,\longrightarrow\,D$ between irreducible smooth complex
projective curves, $f^*E$ is stable for every stable vector bundle $E$ on $D$ if
the homomorphism of topological fundamental groups induced by $f$
$$
f_*\, :\, \pi_1(C,\, x_0)\,\longrightarrow\, \pi_1(D,\, f(x_0))
$$
is surjective.

Theorem \ref{thm1} was stated in \cite{PS} (it is \cite[p.~524, Lemma 3.5(b)]{PS}) without a
complete proof. In the proof of Lemma 3.5(b), which is given in three sentences in
\cite[p.~524]{PS}, it is claimed that the socle of a semistable
bundle descends under any ramified covering map (the first sentence).

\section{Genuinely ramified morphism}

The base field $k$ is assumed to be algebraically closed; there is no restriction on the
characteristic of $k$.

Let $V$ be a vector bundle on an irreducible smooth projective curve $X$ defined over $k$.
If $$0\,=\, V_0\, \subset\, V_1\, \subset\, \cdots \, \subset\, V_{n-1}\, \subset\, V_n\, =\, V$$
is the Harder-Narasimhan filtration of $V$ \cite[p.~16, Theorem 1.3.4]{HL}, then define
$$
\mu_{\rm max}(V)\,:=\, \mu(V_1)\,:=\,\frac{\text{deg}(V_1)}{\text{rk}(V_1)}\ \
\text{ and }\ \ \mu_{\rm min}(V)\,:=\, \mu(V/V_{n-1})\, .
$$
Furthermore, the above subbundle $V_1\, \subset\, V$ is called the \textit{maximal semistable subbundle} of $V$.

If $V$ and $W$ are vector bundles on $X$, and $\beta\,\in\, H^0(X,\, \text{Hom}(V,\, W))\setminus \{0\}$,
then it can be shown that
\begin{equation}\label{a1}
\mu_{\rm max}(W)\, \geq \, \mu_{\rm min}(V)\, .
\end{equation}
Indeed, we have
$$
\mu_{\rm min}(V)\, \leq \, \mu(\beta(V)) \, \leq\, \mu_{\rm max}(W)\, .
$$

\begin{remark}\label{rem1}
Let $f\,:\,X\,\longrightarrow\, Y$ be a nonconstant separable morphism
between irreducible smooth projective curves, and
let $F$ be a semistable vector bundle on $Y$. Then it is known that $f^*F$ is also semistable. Indeed, fixing
a nonconstant separable morphism $h\,:\,Z\,\longrightarrow\, X$, where $Z$
is an irreducible smooth projective curve and $f\circ h$ is Galois, we see
that $(f\circ h)^*F$ is semistable, because its maximal semistable subbundle,
being $\text{Gal}(f\circ h)$ invariant, descends to a subbundle of $F$.
The semistability of $(f\circ h)^*F\,=\, h^*f^*F$ immediately implies that $f^*F$ is semistable.
\end{remark}

\begin{lemma}\label{le1}
Let $f\,:\,X\,\longrightarrow\, Y$ be a nonconstant separable morphism of irreducible smooth projective curves.
Then for any semistable vector bundle $E$ on $X$,
$$\mu_{\rm max}(f_*E)\, \leq\, \mu(E)/{\rm deg}(f)\, .$$

More generally, for any vector bundle $E$ on $X$,
$$\mu_{\rm max}(f_*E)\, \leq\, \mu_{\rm max}(E)/{\rm deg}(f)\, .$$
\end{lemma}

\begin{proof}
Let $E$ be any vector bundle on $X$.
The coherent sheaf $f_*E$ on $Y$ is locally free, because it is torsion-free. We have
\begin{equation}\label{e1}
H^0(Y, \, {\rm Hom}(F,\,f_*E))\, \cong\, H^0(X, \, {\rm Hom}(f^*F,\,E))
\end{equation}
for any vector bundle $F$ on $Y$; see \cite[p.~110]{Ha}.
Setting $F$ in \eqref{e1} to be a semistable
subbundle $V$ of $f_*E$ we see that $H^0(X, \, {\rm Hom}(f^*V,\,E))\, \not=\, 0$.
The pullback $f^*V$ is semistable because $V$ is semistable and $f$ is separable;
see Remark \ref{rem1}.

First take $E$ to be semistable. Hence for any
nonzero homomorphism $\beta\, :\, f^*V\,\longrightarrow\, E$,
\begin{equation}\label{re1}
{\rm deg}(f)\cdot\mu(V) \,=\, \mu(f^*V)\, \leq\, \mu(E)
\end{equation}
(see \eqref{a1}). Now setting $V$ in \eqref{re1} to be the maximal semistable subbundle of
$f_*E$ we conclude that
\begin{equation}\label{e2}
\mu_{\rm max}(f_*E)\, \leq\, \mu(E)/{\rm deg}(f).
\end{equation}

To prove the general (second) statement, for any vector bundle $E$ on $X$, let
$$0\,=\, E_0\, \subset\, E_1\, \subset\, \cdots \, \subset\, E_{n-1}\, \subset\, E_n\, =\, E$$
be the Harder--Narasimhan filtration of $E$ \cite[p.~16, Theorem 1.3.4]{HL}. Consider the filtration of subbundles
\begin{equation}\label{e3}
0\,=\, f_*E_0\, \subset\, f_*E_1\, \subset\, \cdots \, \subset\, f_*E_{n-1}\, \subset\, f_*E_n\, =\, f_*E\, .
\end{equation}
{}From \eqref{e2} we know that
$$
\mu_{\rm max}((f_*E_i)/(f_*E_{i-1})) \,=\, \mu_{\rm max}(f_*(E_i/E_{i-1}))
\, \leq\, \mu(E_i/E_{i-1})/{\rm deg}(f)
$$
\begin{equation}\label{re2}
\leq\, \mu(E_1)/{\rm deg}(f)\,=\,\mu_{\rm max}(E)/{\rm deg}(f)
\end{equation}
for all $1\, \leq\, i\, \leq\, n$. Observe that using \eqref{a1} and the filtration in \eqref{e3} it
follows that
$$
\mu_{\rm max}(f_*E)\,\leq\, {\rm Max}\{\mu_{\rm max}((f_*E_i)/(f_*E_{i-1}))\}_{i=1}^n,
$$
while from \eqref{re2} we have
$$
{\rm Max}\{\mu_{\rm max}((f_*E_i)/(f_*E_{i-1}))\}_{i=1}^n
\, \leq\, \mu_{\rm max}(E)/{\rm deg}(f).$$
Therefore, $\mu_{\rm max}(f_*E)\,\leq\,\mu_{\rm max}(E)/{\rm deg}(f)$, and
this completes the proof.
\end{proof}

The following lemma characterizes the \'etale maps among the separable morphisms.

\begin{lemma}\label{etalevsdegree0}
Let $f\,:\,X\,\longrightarrow\, Y$ be a nonconstant separable morphism between irreducible
smooth projective curves. Then the following three conditions are equivalent:
\begin{enumerate}
\item The map $f$ is \'etale.

\item The degree of $f_*{\mathcal O}_X$ is zero.

\item The vector bundle $f_*{\mathcal O}_X$ is semistable.
 \end{enumerate}
\end{lemma}

\begin{proof}
We have $\mu_{\rm max}(f_*{\mathcal O}_X)\, \geq\, 0$, because ${\mathcal O}_Y\, \subset\,
f_*{\mathcal O}_X$ (see \eqref{e1} and \eqref{a1}). On the other hand, from
Lemma~\ref{le1} it follows that $\mu_{\rm max}(f_*{\mathcal O}_X)\, \leq\, 0$,
so
\begin{equation}\label{y1}
\mu_{\rm max}(f_*{\mathcal O}_X)\, =\, 0\, .
\end{equation}
{}From \eqref{y1} it follows immediately that
$f_*{\mathcal O}_X$ is semistable if and only if $\text{deg}(f_*{\mathcal O}_X)\,=\, 0$.
Therefore, statements (2) and (3) are equivalent.

Let ${\mathcal R}$ be the ramification divisor on $X$ for $f$; define the effective divisor $B\,:=\,f_*{\mathcal R}$
on $Y$. We know that
$$
(\det f_*{\mathcal O}_X)^{\otimes 2}\, =\, {\mathcal O}_Y(-B)
$$
(see \cite[p.~306, Ch.~IV, Ex.~2.6(d)]{Ha}, \cite{Se}). Therefore, $f$ is \'etale (meaning $B\,=\, 0$) if and
only if ${\rm deg}(f_*{\mathcal O}_X) \,=\, 0$. So statements (1) and (2) are equivalent.
\end{proof}

Let $f\,:\,X\,\longrightarrow\, Y$ be a nonconstant separable morphism between irreducible
smooth projective curves.
The algebra structure of ${\mathcal O}_X$ produces an ${\mathcal O}_Y$--algebra structure on
the direct image $f_*{\mathcal O}_X$.

\begin{lemma}\label{maximaletalsubcover}
Let ${\mathcal V}\, \subset\, f_*{\mathcal O}_X$ be the maximal semistable subbundle.
Then ${\mathcal V}$ is a sheaf of ${\mathcal O}_Y$--subalgebras of $f_*{\mathcal O}_X$. 
\end{lemma}

\begin{proof}
The action of ${\mathcal O}_Y$ on $f_*{\mathcal O}_X$ is the standard one. Let
$$
{\mathbf m}\, :\, (f_*{\mathcal O}_X)\otimes (f_*{\mathcal O}_X)\, \longrightarrow\, f_*{\mathcal O}_X
$$
be the ${\mathcal O}_Y$--algebra structure on the direct image $f_*{\mathcal O}_X$
given by the algebra structure of the coherent sheaf ${\mathcal O}_X$. We need to show that
\begin{equation}\label{e4}
{\mathbf m}({\mathcal V}\otimes {\mathcal V})\, \subset\, {\mathcal V}\, ,
\end{equation}
where $\mathcal V$ is the maximal semistable subbundle of $f_*{\mathcal O}_X$.

Since $\mathcal V$ is semistable of degree zero
(see \eqref{y1}), and $\mu_{\rm max}((f_*{\mathcal O}_X)/{\mathcal V})\, <\, 0$,
using \eqref{a1} we conclude that in order to prove \eqref{e4} it suffices to show that
${\mathcal V}\otimes {\mathcal V}$ is semistable of degree zero. Indeed, there
is no nonzero homomorphism from ${\mathcal V}\otimes {\mathcal V}$ to $(f_*{\mathcal O}_X)/{\mathcal V}$,
if ${\mathcal V}\otimes {\mathcal V}$ is semistable of degree zero.

We have $\text{deg}({\mathcal V}\otimes {\mathcal V})\,=\, 0$, because
$\text{deg}({\mathcal V})\,=\, 0$. So ${\mathcal V}\otimes {\mathcal V}$ is semistable if
it does not contain any coherent subsheaf of positive degree. As
$${\mathcal V}\otimes {\mathcal V}\, \subset\, (f_*{\mathcal O}_X)\otimes (f_*{\mathcal O}_X)\, ,$$
if ${\mathcal V}\otimes {\mathcal V}$ contains a subsheaf of positive degree, then
$(f_*{\mathcal O}_X)\otimes (f_*{\mathcal O}_X)$ also contains a subsheaf of positive degree.

Therefore, to prove the lemma it is enough to show that $(f_*{\mathcal O}_X)\otimes (f_*{\mathcal O}_X)$
does not contain any subsheaf of positive degree.

The projection formula, \cite[p.~124, Ch.~II, Ex.~5.1(d)]{Ha}, \cite{Se}, says that
\begin{equation}\label{y3}
(f_*{\mathcal O}_X)\otimes_{{\mathcal O}_Y} (f_*{\mathcal O}_X) \,\cong\, f_*(f^*(f_*{\mathcal O}_X))\,.
\end{equation}
Since ${\mathcal O}_Y\, \subset\, f_*{\mathcal O}_X$, we have
$$
{\mathcal O}_Y\,=\, {\mathcal O}_Y\otimes_{{\mathcal O}_Y} {\mathcal O}_Y\, \subset\,
(f_*{\mathcal O}_X)\otimes_{{\mathcal O}_Y} (f_*{\mathcal O}_X)\, ,
$$
and hence $\mu_{\rm max}((f_*{\mathcal O}_X)\otimes_{{\mathcal O}_Y} (f_*{\mathcal O}_X))\, \geq\, 0$.
Now from \eqref{y3} it follows that
\begin{equation}\label{y2}
\mu_{\rm max}(f_*(f^*(f_*{\mathcal O}_X)))\, \geq\, 0\, .
\end{equation}

Since $f$ is separable, the pullback, by $f$, of a semistable bundle on $Y$ is semistable (see
Remark \ref{rem1}), and consequently
the Harder--Narasimhan filtration of $f^*F$ is the pullback, by $f$, of the
Harder--Narasimhan filtration of $F$. Therefore, from \eqref{y1} it follows that
$$\mu_{\rm max}(f^*(f_*{\mathcal O}_X)) \,=\, 0\, .$$
Now applying the second part of Lemma \ref{le1},
$$0\,=\, \mu_{\rm max}(f^*(f_*{\mathcal O}_X))/\text{deg}(f)
\,\geq \, \mu_{\rm max}(f_*(f^*(f_*{\mathcal O}_X)))\, .$$
This and \eqref{y2} together imply that
$$
\mu_{\rm max}(f_*(f^*(f_*{\mathcal O}_X)))\, =\, 0\, .
$$
Therefore, using \eqref{y3} it follows that
$$
\mu_{\rm max}((f_*{\mathcal O}_X)\otimes (f_*{\mathcal O}_X))\,=\, 0\, .
$$
Hence $(f_*{\mathcal O}_X)\otimes (f_*{\mathcal O}_X)$ does
not contain any subsheaf of positive degree. It was shown earlier that the lemma follows from the
statement that $(f_*{\mathcal O}_X)\otimes (f_*{\mathcal O}_X)$ does
not contain any subsheaf of positive degree.
\end{proof}

\begin{definition}\label{def1}
A nonconstant separable morphism $f\,:\,X\,\longrightarrow\, Y$ between irreducible smooth projective curves is
called {\em genuinely ramified} if 
${\mathcal O}_Y$ is the maximal semistable subbundle of $f_*{\mathcal O}_X$. 
\end{definition}

\begin{proposition}\label{genuinerampi1surj}
Let $f\,:\,X\,\longrightarrow\, Y$ be a nonconstant separable morphism between irreducible smooth projective curves.
Then the following three conditions are equivalent:
\begin{enumerate}
\item The map $f$ is genuinely ramified.

\item The map $f$ does not factor through any nontrivial \'etale cover of $Y$ (in particular, $f$ is not
nontrivial \'etale).

\item The homomorphism between \'etale fundamental groups induced by $f$
$$f_*\, :\, \pi_1^{\rm et}(X)\,\longrightarrow\, \pi_1^{\rm et}(Y)$$
is surjective. 
\end{enumerate}
\end{proposition}

\begin{proof} 
(1)~$\Longrightarrow$~(2): If $f$ factors through a nontrivial \'etale covering $g\,:\,{\widetilde Y}
\,\longrightarrow\, Y$, then $g_*{\mathcal O}_{\widetilde Y}$ is semistable of degree zero (see
Lemma \ref{etalevsdegree0})
and its rank coincides with the degree of $g$. Since
$$g_*{\mathcal O}_{\widetilde Y}\, \subset\, f_*{\mathcal O}_X\, ,$$ this implies that $f$ is not genuinely ramified.

(2)~$\Longrightarrow$~(1): Lemma \ref{maximaletalsubcover} says that the maximal semistable subbundle
${\mathcal V}\, \subset\, f_*{\mathcal O}_X$ is a subalgebra. If $f$ is not genuinely ramified, then
by taking the spectrum of $\mathcal V$ we obtain a separable, possibly ramified, covering map
\begin{equation}\label{g}
g\,:\,{\widetilde Y}\,=\, {\rm Spec}\, {\mathcal V}
\,\longrightarrow\, Y
\end{equation}
whose degree coincides with the rank of $\mathcal V$. We have $g_*{\mathcal O}_{\widetilde Y}\,=\, {\mathcal V}$,
and the inclusion
map ${\mathcal V}\, \hookrightarrow\, f_*{\mathcal O}_X$ defines a map $$h\, :\, X\, \longrightarrow\,
{\widetilde Y}$$ such that
\begin{equation}\label{re3}
g\circ h\,=\, f.
\end{equation}
Since $f$ is separable, from \eqref{re3}
it follows that $g$ is also separable. It can be shown that $g$ is \'etale. To prove this, first
note that $g_*{\mathcal O}_{\widetilde Y}$ is
semistable, because $g_*{\mathcal O}_{\widetilde Y}\,=\, {\mathcal V}$ and
${\mathcal V}$ is semistable. Next, from \eqref{y1} and the semistability of $\mathcal V$ it follows
that $\mu(g_*{\mathcal O}_{\widetilde Y})\,=\,\mu_{\rm max}(g_*{\mathcal O}_{\widetilde Y})
\,=\, 0$. Now Lemma \ref{etalevsdegree0} gives that the map $g$ in \eqref{g} is \'etale.

Since $g$ is \'etale, and \eqref{re3} holds, we conclude that the statement (2) fails. Hence
the statement (2) implies the statement (1).

The equivalence between the statements (2) and (3) follows from the definition of the \'etale fundamental group.
\end{proof}

Let $f\,:\,X\,\longrightarrow\, Y$ be a nonconstant separable morphism between irreducible smooth projective
curves. Let
$$
g\, :\, {\widetilde Y}\,:=\, {\rm Spec}\, {\mathcal V}\, \longrightarrow\, Y
$$
be the \'etale covering corresponding to the maximal semistable subbundle ${\mathcal V}\, \subset\,
f_*{\mathcal O}_X$ (see \eqref{g}; it was shown that the map in \eqref{g} is \'etale). Let
\begin{equation}\label{h}
h\, :\, X\,\longrightarrow\, {\widetilde Y}
\end{equation}
be the morphism given by the inclusion map ${\mathcal V}\, \hookrightarrow\,
f_*{\mathcal O}_X$.

\begin{corollary}\label{cor1}
The map $h$ in \eqref{h} is genuinely ramified.
\end{corollary}

\begin{proof}
Let $\beta \, :\, Z\, \longrightarrow\, {\widetilde Y}$ be an \'etale covering such that
there is a map $$\gamma\, :\, X\,\longrightarrow\, Z$$ satisfying the condition
$\beta\circ\gamma\,=\, h$. Since $(g\circ\beta)\circ\gamma\,=\, f$, we have
\begin{equation}\label{a2}
g_*{\mathcal O}_{\widetilde Y}\, \subset\,
(g\circ\beta)_*{\mathcal O}_Z\, \subset\, f_*{\mathcal O}_X\, ;
\end{equation}
also, we have $\text{deg}((g\circ\beta)_*{\mathcal O}_{\widetilde Y})\,=\, 0$, because $g\circ\beta$
is \'etale (see Lemma \ref{etalevsdegree0}). But ${\mathcal V}\,=\, g_*{\mathcal O}_{\widetilde Y}$
is the maximal semistable subsheaf of $f_*{\mathcal O}_X$. Hence from \eqref{a2} it follows that
$g_*{\mathcal O}_{\widetilde Y}\, =\, (g\circ\beta)_*{\mathcal O}_Z$. This implies that
$\text{deg}(\beta)\,=\,1$. Therefore, from Proposition \ref{genuinerampi1surj} we conclude that
the map $h$ in \eqref{h} is genuinely ramified.
\end{proof}

\section{Properties of genuinely ramified morphisms}

\begin{lemma}\label{genuinerampi1surjcurves}
Let $f\,:\,C\, \longrightarrow\, D$ be a nonconstant separable morphism between irreducible
smooth projective curves. Then the following three conditions are equivalent:
\begin{enumerate}
\item The map $f$ is genuinely ramified. 

\item $\dim H^0(C,\, f^*f_*{\mathcal O}_C)\,=\,1$.

\item The fiber product $C\times _D C$ is connected.
\end{enumerate}
\end{lemma}

\begin{proof}
Let ${\widetilde{C\times_D C}}$ be the normalization of the fiber product $C\times_D C$; it is
a smooth projective curve, but it is not connected unless $f$ is an isomorphism.
We have the commutative diagram 
\begin{equation}\label{d1}
\xymatrix{
\widetilde{C\times_D C} \ar@/_/[ddr]_-{\widetilde{\pi}_1} \ar[dr]^-\nu \ar@/^/[drrr]^-{\widetilde{\pi}_2} & & & \\
& C\times_D C \ar[rr]^-{\pi_2} \ar[d]^-{\pi_1} && C \ar[d]^-f \\
& C \ar[rr]^-f && D
}
\end{equation}
By flat base change \cite[p.~255, Proposition 9.3]{Ha},
\begin{equation}\label{f1}
f^* (f_* {\mathcal O}_C )\,\cong\, {\pi_1}_* (\pi^*_2 {\mathcal O}_C ) \,=\, {\pi_1}_* {\mathcal O}_{C\times_D C}\, .
\end{equation}

(1)~$\Longrightarrow$~(2): Since $f$ is separable, $f^*F$ is semistable if $F$ is so (see
Remark \ref{rem1}), and hence
the maximal semistable subbundle of $f^*f_*{\mathcal O}_C$ is $f^*\mathcal V$, where ${\mathcal V}\,
\subset\, f_*{\mathcal O}_C$ is the maximal semistable subbundle. If $f$ is genuinely ramified, then
the maximal semistable subbundle of $f^*f_*{\mathcal O}_C$ is $f^*{\mathcal O}_D\,=\, {\mathcal O}_C$.
On the other hand,
$$H^0(C,\, (f^*f_*{\mathcal O}_C)/(f^*\mathcal V))\,=\, 0\, ,$$
because $\mu_{\rm max}((f^*f_*{\mathcal O}_C)/(f^*\mathcal V))\, <\, 0$ (see \eqref{a1}).
These together imply that $$\dim H^0(C,\, f^*f_*{\mathcal O}_C)\,=\,1\, ;$$
to see this consider the long exact sequence of cohomologies associated to the short exact sequence
$$
0\, \longrightarrow\, f^*\mathcal V\, \longrightarrow\,f^*f_*{\mathcal O}_C \, \longrightarrow\, 
(f^*f_*{\mathcal O}_C)/(f^*\mathcal V)\, \longrightarrow\,0.
$$

(2)~$\Longleftrightarrow$~(3): From \eqref{f1} it follows that 
\begin{equation}\label{t1}
H^0(C,\, f^*f_*{\mathcal O}_C) \,=\, H^0(C,\, {\pi_1}_*{\mathcal O}_{C\times_D C})
\,=\, H^0(C\times_D C,\, {\mathcal O}_{C\times_D C})\, .
\end{equation}
Consequently, $C\times _D C$ is connected if and only if $\dim H^0(C,\, f^*f_*{\mathcal O}_C)\,=\,1$. 

(3)~$\Longrightarrow$~(1): Assume that $f$ is \textit{not} genuinely ramified. We will prove that $C\times_D C$ is 
not connected.

Let $g\,:\,\widetilde{D}\, \longrightarrow\, D$ be the \'etale cover of $D$ given by ${\rm Spec}\, {\mathcal W}$,
where ${\mathcal W}\, \subset\, f_*{\mathcal O}_C$ is the maximal semistable subbundle (as
in \eqref{g}). The degree of this
covering $g$ is at least two, because $f$ is not genuinely ramified.
To prove that $C\times_D C$ is not connected it suffices to show that ${\widetilde D}\times_D {\widetilde D}$ is not
connected.

The projection $$\gamma\,:\, {\widetilde D}\times_D {\widetilde D}\,\longrightarrow\, {\widetilde D}$$
to the first factor is
evidently the base change of $g\,:\,\widetilde{D}\, \longrightarrow\, D$ to $\widetilde{D}$, and hence
the map $\gamma$ is \'etale. The diagonal ${\widetilde D}\, \hookrightarrow\, {\widetilde D}\times_D {\widetilde D}$
is a connected component of ${\widetilde D}\times_D {\widetilde D}$.
This implies that ${\widetilde D}\times_D {\widetilde D}$ is not connected, because the degree of
$\gamma$ is at least two.
\end{proof} 

\begin{definition}\label{def2}
A nonconstant morphism $f\,:\,C\,\longrightarrow\,D$ between irreducible smooth projective curves
will be called a \textit{separable Galois morphism} if $f$ is separable, and there is a reduced
finite subgroup $\Gamma\, \subset\, \text{Aut}(C)$ such that $D\,=\, C/\Gamma$ and
$f$ is the quotient map $C\,\longrightarrow\,C/\Gamma$. Note that a separable Galois morphism
need not be \'etale. A separable Galois morphism which is genuinely ramified will be called
a \textit{genuinely ramified Galois morphism}.
\end{definition}

\begin{proposition}\label{sumofnegetivedegreelinebundles}
Let $f\,:\,C\,\longrightarrow\,D$ be a separable Galois morphism, of degree $d$,
between irreducible smooth projective curves.
Then $f^* ((f_* {\mathcal O}_C) / {\mathcal O}_D)$
is a coherent subsheaf of ${\mathcal O}^{\oplus (d-1)}_C$.
\end{proposition}

\begin{proof}
The Galois group $\text{Gal}(f)$ of $f$ will be denoted by $\Gamma$.
For any point $x\, \in\, C$, let $$\Gamma_x\, \subset\, \Gamma$$ be the isotropy
subgroup that fixes $x$ for the action of
$\Gamma$ on $C$. A point $(x,\, y)\, \in\, C\times_D C$ is singular if and
only if
$\Gamma_x$ is nontrivial. Note that for any $(x,\, y)\, \in\, C\times_D C$
the two isotropy subgroups $\Gamma_x$ and $\Gamma_y$ are conjugate, because
$y$ lies in the orbit $\Gamma\cdot x$ of $x$.
For any $\sigma\, \in\, \Gamma$, let
\begin{equation}\label{cs}
C_\sigma\, \subset\, C\times_D C
\end{equation}
be the irreducible component given by the image of the map
$$\beta_\sigma\, :\, C\,\longrightarrow\, C\times C\, ,\ \ x\,\longmapsto\, (x,\,\sigma(x))\, ;$$
clearly we have $\beta_\sigma(C)\, \subset\, C\times_D C$.
In this way, the irreducible components of $C\times_D C$ are parametrized by the elements of
the Galois group $\Gamma$. Note that there is a canonical identification
\begin{equation}\label{re8}
C\, \stackrel{\sim}{\longrightarrow}\, C_\sigma
\end{equation}
for every $\sigma\, \in\, \Gamma$.

Let $\widetilde{C\times_D C}$ be the normalization of $C\times_D C$. The maps $\beta_\sigma$,
$\sigma\, \in\, \Gamma$, in \eqref{cs} together produce
an isomorphism
\begin{equation}\label{re7}
C\times\Gamma\,\stackrel{\sim}{\longrightarrow}\, \widetilde{C\times_D C}\, ;
\end{equation}
this map sends any $(y,\, \sigma)\, \in\, C\times\Gamma$ to $(y,\, \sigma(y))$ if $\Gamma_y$ is trivial;
if $\Gamma_y$ is trivial, then $(y,\, \sigma(y))$ is a smooth point of $C\times_D C$ and hence $(y,\, \sigma(y))$
gives a unique point of $\widetilde{C\times_D C}$. Consequently, we have
\begin{equation}\label{e5}
\widetilde{\pi}_{1*} {\mathcal O}_{\widetilde{C\times_D C}}\,=\,{\mathcal O}_C\otimes_k k[\Gamma]\, ,
\end{equation}
where $\widetilde{\pi}_1$ is the projection in \eqref{d1}, and
$k[\Gamma]$ is the group ring. The natural inclusion ${\mathcal O}_{C\times_D C}\,\hookrightarrow\,
\nu_*{\mathcal O}_{\widetilde{C\times_D C}}$, where $\nu$ is the map in \eqref{d1},
induces an injective homomorphism
\begin{equation}\label{e6}
\varphi\, :\, {\pi_1}_* {\mathcal O}_{C\times_D C}\, \hookrightarrow\,
{\pi_1}_* \nu_*{\mathcal O}_{\widetilde{C\times_D C}}\,=\,
\widetilde{\pi}_{1*}{\mathcal O}_{\widetilde{C\times_D C}}\, ,
\end{equation}
where $\pi_1$ and $\widetilde{\pi}_1$ are the maps in \eqref{d1}.

Let
\begin{equation}\label{z2}
\xi\, :\, {\mathcal O}_C \, \longrightarrow\, {\mathcal O}_C\otimes_k k[\Gamma]
\end{equation}
be the composition of homomorphisms
$$
{\mathcal O}_C \, \longrightarrow\, {\pi_1}_* {\mathcal O}_{C\times_D C} \,
\stackrel{\varphi}{\longrightarrow}\, \widetilde{\pi}_{1*}{\mathcal O}_{\widetilde{C\times_D C}}\,=\,
{\mathcal O}_C\otimes_k k[\Gamma]
$$
(see \eqref{e6} and \eqref{e5}). Note that the image $\xi({\mathcal O}_C)$ in \eqref{z2} is a subbundle
of ${\mathcal O}_C\otimes_k k[\Gamma]$, because the section $$\xi(1_C)\, \subset\,
H^0(C,\, \widetilde{\pi}_{1*}{\mathcal O}_{\widetilde{C\times_D C}})\,=\,
k[\Gamma]$$
is nowhere vanishing, where $1_C$ is the constant function $1$ on $C$. There is a trivial
subbundle ${\mathcal E}$ of the trivial bundle ${\mathcal O}_C\otimes_k k[\Gamma]$
$$
{\mathcal O}^{\oplus (d-1)}_C\,=\, {\mathcal E}\, \subset\, {\mathcal O}_C\otimes_k k[\Gamma]
$$
such that
\begin{equation}\label{e7}
{\mathcal E}\oplus \xi({\mathcal O}_C)\,=\, {\mathcal O}_C\otimes_k k[\Gamma]\, .
\end{equation}
To see this, take any point $x\, \in\, C$, and choose a subspace
$$
{\mathcal E}_x\, \subset\, ({\mathcal O}_C\otimes_k k[\Gamma])_x\,=\,
k[\Gamma]
$$
such that $k[\Gamma]\,=\, {\mathcal E}_x\oplus \xi({\mathcal O}_C)_x$; then take
$$
{\mathcal E}\,:=\, {\mathcal O}_C\otimes_k {\mathcal E}_x\, \subset\, 
{\mathcal O}_C\otimes_k k[\Gamma]\, .
$$
This subbundle $\mathcal E$ clearly satisfies the condition in \eqref{e7}.

{}From the decomposition in \eqref{e7} we conclude that $({\mathcal O}_C\otimes_k k[\Gamma])/\xi({\mathcal O}_C)\,=
\, \mathcal E$. Using the isomorphism in \eqref{e5}, the homomorphism $\varphi$ in \eqref{e6} gives a homomorphism
\begin{equation}\label{e8}
\varphi'\, :\, ({\pi_1}_* {\mathcal O}_{C\times_D C})/{\mathcal O}_C \, \longrightarrow\, 
({\mathcal O}_C\otimes_k k[\Gamma])/(\varphi({\mathcal O}_C))\,=\,
({\mathcal O}_C\otimes_k k[\Gamma])/(\xi({\mathcal O}_C))\,=\, {\mathcal E}.
\end{equation}
On the other hand, the isomorphism in \eqref{f1} produces an isomorphism
$$
({\pi_1}_* {\mathcal O}_{C\times_D C})/{\mathcal O}_C \,\cong\,f^* ((f_* {\mathcal O}_C) / {\mathcal O}_D)\, .
$$
Combining this isomorphism with the homomorphism $\varphi'$ in \eqref{e8} we get a homomorphism
$$
f^* ((f_* {\mathcal O}_C) / {\mathcal O}_D)\, \longrightarrow\,\mathcal E\,=\, {\mathcal O}^{\oplus (d-1)}_C\, .
$$
This homomorphism is clearly an isomorphism over the nonempty open subset of $C$ where $f$ is \'etale.
\end{proof}

Note that $f^* ((f_* {\mathcal O}_C) / {\mathcal O}_D)\,=\, (f^*f_* {\mathcal O}_C)/{\mathcal O}_C$; but we use
$f^* ((f_* {\mathcal O}_C) / {\mathcal O}_D)$ due to the relevance of $(f_* {\mathcal O}_C)/{\mathcal O}_D$.

Let
$$f\,:\,C\,\longrightarrow\,D$$ be a genuinely ramified Galois morphism, of degree $d$,
between irreducible smooth projective curves; see Definition \ref{def2}. As before, the Galois group $\text{Gal}(f)$
will be denoted by $\Gamma$, so we have $$\# \Gamma\,=\, d\, .$$
Assume that $d\, > \, 1$.

As in \eqref{cs}, the irreducible component of $C\times_D C$ corresponding
to $\sigma\, \in\, \Gamma$ will be denoted by $C_\sigma$.

The following lemma formulated in the above set-up 
will be used in proving a variation of Proposition \ref{sumofnegetivedegreelinebundles}.

\begin{lemma}\label{lemord}
There is an ordering of the elements of $\Gamma$
$$
\Gamma\, =\, \{\gamma_1,\, \gamma_2,\, \cdots,\, \gamma_d\}
$$
and a self-map
$$
\eta\, :\, \{1,\, 2,\, \cdots ,\, d\}\, \longrightarrow\, \{1,\, 2,\, \cdots ,\, d\}
$$
such that
\begin{enumerate}
\item $\gamma_1\,=\, e$ (the identity element of $\Gamma$),

\item $\eta(1)\, =\, 1$,

\item $\eta(j) \, <\, j$ for all $j\, \in\, \{2,\, \cdots ,\, d\}$, and

\item $C_{\gamma_j}\bigcap C_{\gamma_{\eta(j)}}\, \not=\, \emptyset$ (see \eqref{cs} for notation).
\end{enumerate}
\end{lemma}

\begin{proof}
Set $\Gamma_0\, :=\, \gamma_1$ to be the identity element $e\, \in\, \Gamma$; also, set
$\eta(1)\,=\, 1$. Set $N_0\, =\,1$.

Let $$\Gamma_1\, \subset\, \Gamma$$ be the subset consisting of all $\gamma\, \not=\, e$ such that
the action of $\gamma$ on $C$ has a fixed point. Therefore, $\Gamma_1$ consists of all $\gamma\, \not=\, e$ such that
the irreducible component $C_\gamma\,\subset\, C\times_D C$ intersects the component $C_e\,=\, C_{\gamma_1}$.
We note that $\Gamma_1$ is nonempty, because otherwise
$C_{\gamma_1}$ would be a connected component of $C\times_D C$, while from Lemma
\ref{genuinerampi1surjcurves}(3) we know that $C\times_D C$ is connected;
recall that $\Gamma\, \not=\,\{e\}$ and $f$ is genuinely ramified.

If $\# \Gamma_1\,=\, N_1-1\, =\, N_1-N_0$, set $\gamma_j\, \in\, \Gamma$, $2\, \leq\, j\, \leq\, N_1$, to be
distinct elements of $\Gamma_1$ in an arbitrary order. Set
$$
\eta(j)\, =\, 1
$$
for all $2\, \leq\, j\, \leq\, N_1$.

If $\Gamma_1\bigcup\Gamma_0\, \not=\, \Gamma$,
let $$\Gamma_2\, \subset\, \Gamma\setminus (\Gamma_1\cup \Gamma_0)$$
be the subset consisting of all $\gamma\, \in\, \Gamma\setminus (\Gamma_1\bigcup \Gamma_0)$ such that the irreducible
component
$$C_\gamma\,\subset\, C\times_D C$$ intersects the component $C_\sigma$ for some $\sigma\, \in\, \Gamma_1$.
Note that such a component $C_\gamma$ does not intersect $C_{\gamma_1}$,
because in that case we would have $\gamma\, \in\, \Gamma_1$.

If $\# \Gamma_2\,=\, N_2-N_1$, set $\gamma_j\, \in\, \Gamma$, $N_1+1\, \leq\, j\, \leq\, N_2$,
to be distinct elements of $\Gamma_2$ in an arbitrary order. For every $N_1+1\, \leq\, j\, \leq\, N_2$, set 
$$
\eta(j)\, \in\, \, \{2,\, \cdots ,\, N_1\}
$$
such that the component $C_{\gamma_j}\,\subset\, C\times_D C$
intersects the component $C_{\gamma_{\eta(j)}}$; the above definition of $\Gamma_2$
ensures that such a $\eta(j)$ exists. If there are
more than one $m\, \in\, \{2,\, \cdots ,\, N_1\}$ such that
$C_{\gamma_j}\,\subset\, C\times_D C$ intersects the component $C_{\gamma_m}$, then choose
$\eta(j)$ arbitrarily from them.

Now inductively define
$$
\Gamma_n\, \subset\, \Gamma\setminus (\bigcup_{i=0}^{n-1} \Gamma_i)\, ,
$$
if $\Gamma_n\, \not=\, \emptyset$,
to be the subset consisting of all $\gamma\, \in\, \Gamma\setminus (\bigcup_{i=0}^{n-1} \Gamma_i)$ such that the
irreducible component $C_\gamma\,\subset\, C\times_D C$ intersects the component $C_\sigma$ for some
$\sigma\, \in\, \Gamma_{n-1}$. Note that such a component $C_\gamma$ does not intersect
$\bigcup_{i=0}^{n-2} \Gamma_i$, because in that case $\gamma\, \in\, \bigcup_{i=0}^{n-1} \Gamma_i$.

If $\# \Gamma_n\,=\, N_n- \sum_{i=0}^{n-1}\# \Gamma_i\,=\, N_n-N_{n-1}$, set $\gamma_j\, \in\,
\Gamma$, $N_{n-1}+1\, \leq\, j\, \leq\, N_n$,
to be distinct elements of $\Gamma_n$ in an arbitrary order. For $N_{n-1}+1\, \leq\, j\, \leq\,
N_n$, set
$$
\eta(j)\, \in\, \, [N_{n-2}+1,\, N_{n-1}]\,=\,
\{1+\sum_{i=0}^{n-2}\# \Gamma_i,\, \cdots ,\, \sum_{i=0}^{n-1}\# \Gamma_i\}
$$
such that the component $C_{\gamma_j}\,\subset\, C\times_D C$
intersects the component $C_{\gamma_{\eta(j)}}$. If $C_{\gamma_j}$ intersects more than
one such component, choose $\eta(j)$ to be any one from them, as before.

Since $\Gamma$ is a finite group, we have $\Gamma_n\, =\,\emptyset$ for all $n$ sufficiently large. Set
$$
{\mathbb S}\,=\, \sum_{i=0}^\infty \# \Gamma_i\, =\, \text{Max}_{i\geq 0}\{N_i\}\, .
$$
Note that
$$
\bigcup_{i=1}^{\mathbb S} C_{\gamma_i}\, \subset\, C\times_D C
$$
is the connected component of $C\times_D C$ containing $C_{\gamma_1}$. Hence from 
Lemma \ref{genuinerampi1surjcurves}(3) we know that $$\bigcup_{i=1}^{\mathbb S} C_{\gamma_i}\,=\, C\times_D C .$$
In other words, we have ${\mathbb S}\ =\, d\, :=\, \# \Gamma$. This completes the proof of the lemma.
\end{proof}

\begin{proposition}\label{sumofnegetivedegreelinebundles1}
Let $f\,:\,C\,\longrightarrow\,D$ be a genuinely ramified Galois morphism, of degree $d$,
between irreducible smooth projective curves. Then
$$f^* ((f_* {\mathcal O}_C) / {\mathcal O}_D)\, \subset\, \bigoplus_{i=1}^{d-1} {\mathcal L}_i\, ,$$
where each ${\mathcal L}_i$ is a line bundle on $D$ of negative degree. 
\end{proposition}

\begin{proof}
As in Lemma \ref{lemord}, the Galois group $\text{Gal}(f)$ is denoted by $\Gamma$.
The ordering in Lemma
\ref{lemord} of the elements of $\Gamma$ produces an isomorphism of $k[\Gamma]$ with $k^{\oplus d}$.
Consequently, from \eqref{e5} we have
\begin{equation}\label{z1}
\widetilde{\pi}_{1*} {\mathcal O}_{\widetilde{C\times_D C}}\,=\,
{\mathcal O}_C\otimes_k k[\Gamma]\,=\, {\mathcal O}^{\oplus d}_C\, .
\end{equation}

Let
$$
\Phi\, :\, \widetilde{\pi}_{1*} {\mathcal O}_{\widetilde{C\times_D C}}\,=\, {\mathcal O}^{\oplus d}_C
\, \longrightarrow\, {\mathcal O}^{\oplus d}_C \,=\,
\widetilde{\pi}_{1*} {\mathcal O}_{\widetilde{C\times_D C}}
$$
be the homomorphism defined by
\begin{equation}\label{re4}
(f_1,\, f_2,\, \cdots, \, f_d)\, \longmapsto\,
(f_1-f_{\eta(1)},\, f_2-f_{\eta(2)},\, \cdots, \, f_d-f_{\eta(d)})\, ,
\end{equation}
where $\eta$ is the map in Lemma \ref{lemord};
more precisely, the $i$-th
component of $\Phi(f_1,\, f_2,\, \cdots, \, f_d)$ is $f_i-f_{\eta(i)}$. It is straightforward to
check that
$$
{\mathcal F}\, :=\, \Phi({\mathcal O}^{\oplus d}_C)\, \subset\,
{\mathcal O}^{\oplus d}_C \,=\,
\widetilde{\pi}_{1*} {\mathcal O}_{\widetilde{C\times_D C}}
$$
is a trivial subbundle of rank $d-1$; the first component of $\Phi(f_1,\, f_2,\, \cdots, \, f_d)$
vanishes identically, because $\eta(1)\,=\,1$. More precisely, we have
\begin{equation}\label{zf2}
{\mathcal F}\,=\, {\mathcal O}^{\oplus (d-1)}_C\, \subset\, {\mathcal O}^{\oplus d}_C
\,=\, \widetilde{\pi}_{1*} {\mathcal O}_{\widetilde{C\times_D C}}\, ,
\end{equation}
where ${\mathcal O}^{\oplus (d-1)}_C$ is the subbundle of ${\mathcal O}^{\oplus d}_C$
spanned by all $(f_1,\, f_2,\, \cdots, \, f_d)$ such that $f_1\,=\,0$.

{}From \eqref{zf2} it follows immediately that
\begin{equation}\label{cf2}
\widetilde{\pi}_{1*} {\mathcal O}_{\widetilde{C\times_D C}}\,=\,
{\mathcal O}^{\oplus d}_C\,=\, {\mathcal F}\oplus \xi({\mathcal O}_C)\, ,
\end{equation}
where $\xi({\mathcal O}_C)$ is the subbundle of ${\mathcal O}^{\oplus d}_C\,=\, 
{\mathcal O}_C\otimes_k k[\Gamma]$ in \eqref{z2} (see \eqref{z1}).

In \eqref{d1} we have $\widetilde{\pi}_1\,=\,\pi_1\circ\nu$, and hence,
as in \eqref{e6},
 there is a natural homomorphism
\begin{equation}\label{re5}
\varphi\, :\,
{\pi_1}_* {\mathcal O}_{C\times_D C}\, \hookrightarrow\, \widetilde{\pi}_{1*} {\mathcal O}_{\widetilde{C\times_D C}}
\end{equation}
which is an isomorphism over the open subset of $C$ where $f$ is \'etale. Therefore, from
\eqref{f1} and \eqref{cf2} we get an injective homomorphism
of coherent sheaves
\begin{equation}\label{cf3}
\Psi\, :\, f^* ((f_* {\mathcal O}_C) / {\mathcal O}_D)\,\longrightarrow \, {\mathcal F}\,=\,
{\mathcal O}^{\oplus (d-1)}_C\, ;
\end{equation}
it is similar to \eqref{e8}, except that now the direct summand ${\mathcal F}$ is chosen carefully
(it was $\mathcal E$ in \eqref{e8}).
Note that since $\text{rk}(f^* ((f_* {\mathcal O}_C) / {\mathcal O}_D))\,=\, d-1\,=\,
\text{rk}({\mathcal O}^{\oplus (d-1)}_C)$, the homomorphism $\Psi$ in \eqref{cf3} is generically
an isomorphism, because it is an
injective homomorphism of coherent sheaves. More precisely, $\Psi$ is an isomorphism over the open subset
of $C$ where the map $f$ is \'etale.

Consider the map $\eta$ in Lemma \ref{lemord}. For every $1\,\leq\, i\, \leq\, d-1$, choose a point
\begin{equation}\label{zi}
z_i\, \in\, C_{\gamma_{i+1}}\bigcap C_{\gamma_{\eta(i+1)}}\, ;
\end{equation}
this is possible because
the fourth property in Lemma \ref{lemord} says that the intersection $C_{\gamma_{i+1}}\bigcap
C_{\gamma_{\eta(i+1)}}$ is nonempty. Recall from \eqref{re8} that
$C$ is identified with $C_{\gamma_{i+1}}$. The point $z_i\, \in\, C_{\gamma_{i+1}}$ in \eqref{zi}
will be considered as a point of $C$ using this identification. Let
$$
{\mathcal L}_i\,:=\, {\mathcal O}_C(-z_i)
$$
be the line bundle corresponding to the point $z_i\, \in\, C$.

For every $1\,\leq\, i\, \leq\, d-1$, let
\begin{equation}\label{pj}
P_i\, :\, {\mathcal O}^{\oplus (d-1)}_C\, \longrightarrow\, {\mathcal O}_C
\end{equation}
be the natural projection to the $i$-th factor.

Consider the composition of homomorphisms $P_i\circ\Psi$, where $P_i$ and
$\Psi$ are constructed in \eqref{pj} and \eqref{cf3} respectively. It can be shown that
$P_i\circ\Psi$ vanishes when restricted to the
point $z_i$ in \eqref{zi}. To see this, for any $1\,\leq\, j\, \leq\, d$, let
$$
\widehat{P}_j\, :\, {\mathcal O}^{\oplus d}_C\, \longrightarrow\, {\mathcal O}_C
$$
be the natural projection to the $j$-th factor. Recall the homomorphism $\Phi$ constructed
in \eqref{re4}. If $(f_1,\, f_2,\, \cdots, \, f_d)$ in \eqref{re4} actually lies in the image
of ${\pi_1}_* {\mathcal O}_{C\times_D C}$ by the inclusion map $\varphi$ in \eqref{re5}, then from
\eqref{zi} we have
\begin{equation}\label{re6}
(\widehat{P}_{i+1}\circ \Phi)(f_1,\, f_2,\, \cdots, \, f_d)(z_i,\, \gamma_{i+1}) \,=\,
f_{i+1}(z_i,\, \gamma)-f_{\eta(i+1)}(z_i,\, \gamma_{\eta(i+1)})\,=\, 0\, ,
\end{equation}
where $(z_i,\, \gamma_{i+1})\, \in\, C\times \Gamma \,=\, \widetilde{C\times_D C}$
(see \eqref{re7}) and the same for $(z_i,\, \gamma_{\eta(i+1)})$; note that from \eqref{zi} it follows that
the point in $C$ corresponding to $z_i\, \in\, C_{\gamma_{\eta(i+1)}}$
(see \eqref{zi}) by the identification $C\,\stackrel{\sim}{\longrightarrow}\, C_{\gamma_{\eta(i+1)}}$
in \eqref{re8} coincides with the point corresponding to $z_i\, \in\, C_{\gamma_{i+1}}$ (the element
$\gamma^{-1}_{i+1}\gamma_{\eta(i+1)}\, \in\, \Gamma$ fixes this point of $C$).
To clarify, there is a slight abuse of notation in \eqref{e6} in the following sense:
sections of $\widetilde{\pi}_{1*} {\mathcal O}_{\widetilde{C\times_D C}}$ over an open subset
$U\, \subset\, C$ are identified with function on $\widetilde{\pi}^{-1}(U)$. So
$(f_1,\, f_2,\, \cdots, \, f_d)$ in \eqref{re6} is considered as a function on $\widetilde{\pi}^{-1}(U)$; the above
condition that $(f_1,\, f_2,\, \cdots, \, f_d)$ in \eqref{re6} lies in the image
of ${\pi_1}_* {\mathcal O}_{C\times_D C}$ by the inclusion $\varphi$ map in \eqref{re5} means that
$(f_1,\, f_2,\, \cdots, \, f_d)$ coincides with $\widehat{f}\circ \nu$ for some function $\widehat{f}$ on
$\pi^{-1}_1(U)$, where $\nu$ is the map in \eqref{d1}. Now from \eqref{re6} it follows that
$P_i\circ\Psi$ vanishes when restricted to the point $z_i\, \in\, C$.

Since $P_i\circ\Psi$ vanishes when restricted to the point $z_i$, we have
\begin{equation}\label{cf4}
P_i\circ\Psi(f^* ((f_* {\mathcal O}_C) / {\mathcal O}_D))\, \subset\, {\mathcal L}_i\,=\, {\mathcal O}_C(-z_i)
\, \subset\, {\mathcal O}_C\, .
\end{equation}
{}From \eqref{cf3} and \eqref{cf4} it follows immediately that
$$
f^* ((f_* {\mathcal O}_C)/{\mathcal O}_D)\, \hookrightarrow\,\bigoplus_{i=1}^{d-1} {\mathcal L}_i\, .
$$
Since $\text{deg}({\mathcal L}_i)\,=\, -1$, the proof of the proposition is complete.
\end{proof}

\section{Pullback of stable bundles and genuinely ramified maps}

\begin{lemma}\label{negetiveslope}
Let $f\,:\,C\,\longrightarrow\, D$ be a genuinely ramified morphism
between irreducible smooth projective curves. 
Let $V$ be a semistable vector bundle on $D$. Then 
$$
\mu_{\rm max} ( V\otimes ((f_*{\mathcal O}_C)/{\mathcal O}_D)) \, <\, \mu (V)\, .
$$
\end{lemma}

\begin{proof}
First assume that the map $f$ is Galois. Take the line bundles ${\mathcal L}_i$,
$1\, \leq\, i\, \leq\, d-1$, in
Proposition \ref{sumofnegetivedegreelinebundles1}, where $d\,=\, \text{deg}(f)$. Then from
Proposition \ref{sumofnegetivedegreelinebundles1} we have
$$
\mu_{\rm max} ( V\otimes ((f_*{\mathcal O}_C)/{\mathcal O}_D))\, \leq\,
\mu_{\rm max} (V\otimes (\bigoplus_{i=1}^{d-1} {\mathcal L}_i))\, \leq\,
{\rm \max}\{\mu(V\otimes{\mathcal L}_i)\}_{i=1}^{d-1}\, ,
$$
because $V\otimes{\mathcal L}_i$ is semistable. On the other hand,
$$
\mu(V\otimes{\mathcal L}_i)\, <\, \mu(V)\, ,
$$
because $\text{deg}({\mathcal L}_i)\, <\, 0$. Combining these, we have
$$
\mu_{\rm max} ( V\otimes ((f_*{\mathcal O}_C)/{\mathcal O}_D)) \, <\, \mu (V)\, ,
$$
giving the statement of the proposition.

If the map $f$ is not Galois, consider the smallest Galois extension
\begin{equation}\label{p1}
F\,:\,\widehat{C}\,\longrightarrow\, D
\end{equation}
such that there is a morphism $\widehat{f}\,:\,\widehat{C}\,\longrightarrow\, C$ for which
\begin{equation}\label{q1}
f\circ\widehat{f}\, =\, F\, .
\end{equation}
Note that $\widehat{C}$ is irreducible and smooth, and $F$ is separable. From \eqref{q1} it follows that
\begin{equation}\label{q2}
f_*{\mathcal O}_C\, \subset\, F_*{\mathcal O}_{\widehat C}\, .
\end{equation}

First assume that the map $F$ in \eqref{p1} is genuinely ramified. From
\eqref{q2} it follows that
\begin{equation}\label{s1}
(f_*{\mathcal O}_C)/{\mathcal O}_D \,\subset\, ({F}_*{\mathcal 
O}_{\widehat C})/{\mathcal O}_D\, .
\end{equation}
Since $F$ is Galois, from Proposition \ref{sumofnegetivedegreelinebundles1} we know that
$({F}_*{\mathcal O}_{\widehat C})/{\mathcal O}_D$ is contained 
in a direct sum of line bundles of negative degree. Hence the subsheaf
$(f_*{\mathcal O}_C)/{\mathcal O}_D$ in \eqref{s1} is also contained
in a direct sum of line bundles of negative degree.
This implies that
$$
\mu_{\rm max} ( V\otimes ((f_*{\mathcal O}_C)/{\mathcal O}_D)) \, <\, \mu (V)\, ,
$$
giving the statement of the proposition.

Therefore, we now assume that $F$ is not genuinely ramified. Let
$$
(F_*{\mathcal O}_{\widehat C})_1\, \subset\,
F_*{\mathcal O}_{\widehat C}
$$
be the maximal semistable subbundle. Let
\begin{equation}\label{gl}
g\,:\, \widehat{D}\,\longrightarrow\, D
\end{equation}
be the \'etale cover defined by the spectrum of
the bundle $(F_*{\mathcal O}_{\widehat C})_1$ of ${\mathcal O}_D$--algebras (see Lemma \ref{maximaletalsubcover});
that the map $g$ in \eqref{gl} is \'etale follows from Lemma \ref{etalevsdegree0} and \eqref{y1}, because
$g_*{\mathcal O}_{\widehat{D}}\,=\, (F_*{\mathcal O}_{\widehat C})_1$ and
$(F_*{\mathcal O}_{\widehat C})_1$ is semistable.
We note that the Galois group $\text{Gal}(F)$ for $F$ acts naturally on $F_*{\mathcal O}_{\widehat C}$,
and this action of $\text{Gal}(F)$ preserves the subbundle $(F_*{\mathcal O}_{\widehat C})_1$; indeed, this
follows from the uniqueness of the maximal semistable subbundle
$(F_*{\mathcal O}_{\widehat C})_1$. Therefore, $\text{Gal}(F)$ acts
on $\widehat{D}$, and the map $g$ in \eqref{gl} is $\text{Gal}(F)$--equivariant for the trivial
action of $\text{Gal}(F)$ on $D$. Consequently, the covering $g$ in \eqref{gl} is Galois.

Consider the following commutative diagram
\begin{equation}\label{lcd}
\xymatrix{
\widehat{C} \ar@/_/[ddr]_-{\widehat{f}} \ar[dr]^-h \ar@/^/[drrr]^-{\widehat{g}} & & & \\
& C\times_D\widehat{D} \ar[rr]^-{\pi_2} \ar[d]^-{\pi_1} && \widehat{D} \ar[d]^-g \\
& C \ar[rr]^-f && D
}
\end{equation}
The existence of the map $h$ in \eqref{lcd} is evident. The map $f$ being genuinely ramified,
it follows from Lemma \ref{genuinerampi1surj} that the homomorphism between \'etale fundamental groups
$$f_*\, :\, \pi_1^{\rm et}(C)\,\longrightarrow\, \pi_1^{\rm et}(D)$$ induced
by $f$ is surjective. This implies that the fiber product $C\times_D\widehat{D}$ is connected.
The diagram in \eqref{lcd} should not be confused with the one in \eqref{d1} --- in \eqref{lcd},
$C\times_D\widehat{D}$ is smooth as $g$ is \'etale.

We will prove that the map $\widehat{g}$ in \eqref{lcd} is genuinely ramified and Galois. For this,
first recall the earlier observation that $\text{Gal}(F)$ acts on $\widehat{D}$. The map $\widehat{g}$ is evidently
equivariant for the actions of $\text{Gal}(F)$ on $\widehat{C}$ and $\widehat{D}$.
This immediately implies that the map $\widehat{g}$ is Galois. From
Corollary \ref{cor1} it follows that $\widehat{g}$ is genuinely ramified.

We will next prove that $\widehat{C}\times_D\widehat{C}$ is a disjoint union of curves isomorphic to
$\widehat{C}\times_{\widehat{D}}
\widehat{C}$. For this, first note that $\widehat{C}\times_D\widehat{C}$ maps to
$\widehat{D}\times_D\widehat{D}$, and the curve $\widehat{D}\times_D\widehat{D}$ is a disjoint union of copies of 
$\widehat{D}$ as $g$ is \'etale Galois. The component of $\widehat{C}\times_D\widehat{C}$ lying over any of these
copies of $\widehat{D}$ is 
isomorphic to $\widehat{C}\times_{\widehat{D}}\widehat{C}$, and therefore
$\widehat{C}\times_D\widehat{C}$ is a disjoint union of curves isomorphic to
$\widehat{C}\times_{\widehat{D}}\widehat{C}$.

{}From \eqref{q2} we have
\begin{equation}\label{re11}
F^*f_*{\mathcal O}_C\, \subset\, F^*F_*{\mathcal O}_{\widehat C}\, .
\end{equation}
Since $\widehat{C}\times_D\widehat{C}$ is a disjoint union of curves isomorphic to
$\widehat{C}\times_{\widehat{D}}\widehat{C}$, from \eqref{t1} it follows that 
$F^*F_*{\mathcal O}_{\widehat C}$ is a direct sum of copies of $\widehat{g}^*\widehat{g}_*{\mathcal O}_{\widehat C}$.
It was shown above that $\widehat{g}$ is genuinely ramified and Galois. So from
Proposition \ref{sumofnegetivedegreelinebundles1} we know that
$\widehat{g}^*((\widehat{g}_*{\mathcal O}_{\widehat C})/{\mathcal O}_{\widehat{D}})$ is
contained in a direct sum of line bundles of negative degree. 

Since $\widehat{g}$ is genuinely ramified, we know from Lemma \ref{genuinerampi1surjcurves}, Remark \ref{rem1}
and \eqref{y1} that
$$
{\mathcal O}_{\widehat C}\,=\, H^0({\widehat C},\, \widehat{g}^*\widehat{g}_*{\mathcal O}_D)
\otimes {\mathcal O}_{\widehat C}
$$
is the maximal semistable subbundle of $\widehat{g}^*\widehat{g}_*{\mathcal O}_D$. Since
$F^*F_*{\mathcal O}_{\widehat C}$ is a direct sum of copies of $\widehat{g}^*\widehat{g}_*{\mathcal O}_{\widehat C}$,
this implies that
\begin{equation}\label{re9}
H^0({\widehat C},\, F^*F_*{\mathcal O}_{\widehat C}) \otimes {\mathcal O}_{\widehat C}\, \subset\,
F^*F_*{\mathcal O}_{\widehat C}
\end{equation}
is the maximal semistable subbundle. On the other hand, we have
$F^*f_*{\mathcal O}_C\,=\, \widehat{f}^*f^*f_*{\mathcal O}_C$. So from
Lemma \ref{genuinerampi1surjcurves}, Remark \ref{rem1}
and \eqref{y1} we know that
\begin{equation}\label{re10}
{\mathcal O}_{\widehat C}\,=\,
H^0({\widehat C},\, F^*f_*{\mathcal O}_C) \otimes {\mathcal O}_{\widehat C} \, \subset\,
F^*f_*{\mathcal O}_C
\end{equation}
is the maximal semistable subbundle.

Consider the inclusion homomorphism in \eqref{re11}. From \eqref{re9} and \eqref{re10} we conclude that
using this homomorphism, the
quotient $F^*((f_*{\mathcal O}_C)/{\mathcal O}_D)$ is contained in
\begin{equation}\label{u1}
F^*F_*{\mathcal O}_{\widehat C}/(H^0({\widehat
C},\, F^*F_*{\mathcal O}_{\widehat C})\otimes {\mathcal O}_{\widehat C})\, .
\end{equation}
Since $F^*F_*{\mathcal O}_{\widehat C}$ is a direct sum of copies of
$\widehat{g}^*\widehat{g}_*{\mathcal O}_{\widehat C}$,
the vector bundle in \eqref{u1} is isomorphic to a direct sum of copies of
$\widehat{g}^*((\widehat{g}_*{\mathcal O}_{\widehat C})/{\mathcal O}_{\widehat{D}})$.

It was shown above that
$\widehat{g}^*((\widehat{g}_*{\mathcal O}_{\widehat C})/{\mathcal O}_{\widehat{D}})$ is
contained in a direct sum of line bundles of negative degree. Therefore, the vector bundle in \eqref{u1} is
contained in a direct sum of line bundles of negative degree. Consequently, the subsheaf
$$
F^*((f_*{\mathcal O}_C)/{\mathcal O}_D)\, \subset\,
F^*F_*{\mathcal O}_{\widehat C}/(H^0({\widehat
C},\, F^*F_*{\mathcal O}_{\widehat C})\otimes {\mathcal O}_{\widehat C})
$$
is also contained in a direct sum of line bundles of negative degree.

Since $F^*(f_*{\mathcal O}_C)/{\mathcal O}_D)$ is 
contained in a direct sum of line bundles of negative degree, we conclude that
$$
\mu_{\rm max} ((F^*V)\otimes (F^*((f_*{\mathcal O}_C)/{\mathcal O}_D))) \, <\, \mu (F^*V)\,;
$$
note that $F^*V$ is semistable by Remark \ref{rem1} as $F$ is separable. From this it follows that
$$
\mu_{\rm max} (V\otimes ((f_*{\mathcal O}_C)/{\mathcal O}_D))\,=\,
\mu_{\rm max}(F^*(V\otimes ((f_*{\mathcal O}_C)/{\mathcal O}_D)))/\text{deg}(F)
$$
$$
<\, \mu (F^*V)/\text{deg}(F)\,=\, \mu (V)\, ,
$$
because $F^*V$ is semistable. This completes the proof.
\end{proof}

\begin{remark}
When the characteristic of the base field $k$ is zero, the tensor product of two semistable
bundles remains semistable \cite[p.~285, Theorem 3.18]{RR}. We note that Lemma \ref{negetiveslope}
is a straight-forward consequence of it, provided the characteristic of $k$ is zero.
\end{remark}

\begin{lemma} \label{lemma3.8}
Let $f\,:\,C\,\longrightarrow\, D$ be a genuinely ramified morphism between
irreducible smooth projective curves. 
Let $V$ and $W$ be two semistable vector bundles on $D$ with $$\mu(V)\,=\, \mu(W)\, .$$ Then 
$$H^0(D,\, {\rm Hom}(V,\,W)) \,=\, H^0(C,\, {\rm Hom}(f^*V,\, f^*W))\, .$$ 
\end{lemma}

\begin{proof}
Using the projection formula, and the fact that $f$ is a finite map, we have
$$
H^0(C,\, {\rm Hom}(f^*V,\, f^*W)) \,\cong\,H^0(D,\, f_*{\rm Hom}(f^*V,\, f^*W))
\,\cong\,H^0(D,\, f_*f^*{\rm Hom}(V,\, W))
$$
\begin{equation}\label{j1}
\,\cong\,H^0(D,\, {\rm Hom}(V,\, W)\otimes f_*{\mathcal O}_C)
\,\cong\,H^0(D,\, {\rm Hom}(V,\, W\otimes f_*{\mathcal O}_C))\, .
\end{equation}
Let
$$
0\,=\, B_0\, \subset\, B_1\, \subset\, \cdots \, \subset\, B_{m-1}\, \subset\, B_m\,=\,
W\otimes ((f_*{\mathcal O}_C)/{\mathcal O}_D)
$$
be the Harder--Narasimhan filtration of $W\otimes ((f_*{\mathcal O}_C)/{\mathcal O}_D)$
\cite[p.~16, Theorem 1.3.4]{HL}. Since $W$ is semistable, and $f$ is genuinely ramified, from
Lemma~\ref{negetiveslope} we know that 
$$
\mu(B_i/B_{i-1})\, \leq\, \mu(B_1)\, =\, \mu_{\rm max}(W\otimes ((f_*{\mathcal O}_C)/{\mathcal O}_D))
\, <\, \mu(W)
$$
for all $1\, \leq\, i\, \leq\, m$. In view this and the given condition that $\mu(V)\,=\, \mu(W)$, from
\eqref{a1} we conclude that
$$
H^0(D,\, {\rm Hom}(V,\, B_i/B_{i-1}))\,=\, 0
$$
for all $1\, \leq\, i\, \leq\, m$; note that both $V$ and $B_i/B_{i-1}$ are semistable. This implies that
$$
H^0(D,\, {\rm Hom}(V,\, W\otimes ((f_*{\mathcal O}_C)/{\mathcal O}_D)))\,=\ 0\, .
$$
Consequently, we have
$$
H^0(D,\, {\rm Hom}(V,\, W\otimes f_*{\mathcal O}_C))\,=\,
H^0(D,\, {\rm Hom}(V,\, W))
$$ 
by examining the exact sequence
$$
0 \,\longrightarrow\,{\rm Hom}(V,\, W)\,\longrightarrow\, {\rm Hom}(V,\, W\otimes f_*{\mathcal O}_C)
\,\longrightarrow\, {\rm Hom}(V,\, W\otimes ((f_*{\mathcal O}_C)/{\mathcal O}_D))\,\longrightarrow\, 0\, .
$$
{}From this and \eqref{j1} it follows that
$$H^0(C,\, {\rm Hom}(f^*V,\, f^*W))\,=\, H^0(D,\, {\rm Hom}(V,\,W))\, .$$
This completes the proof.
\end{proof}

\begin{theorem}\label{thm1}
Let $f\,:\,C\,\longrightarrow\, D$ be a genuinely ramified morphism
between irreducible smooth projective curves. Let $V$ be a stable vector
bundle on $D$. Then the pulled back vector bundle $f^*V$ is also stable.
\end{theorem}

\begin{proof}
Consider the Galois extension $F\,:\,\widehat{C}\,\longrightarrow\, D$ and the diagram in
\eqref{lcd}. Since $V$ is stable, from Lemma~\ref{lemma3.8} it follows that $f^*V$ is simple. 
As $V$ is semistable, it follows that $g^*V$ is also semistable, where $g$
is the map in \eqref{lcd}. Let $$E\, \subset\, g^*V$$ be the
unique maximal polystable subbundle with $\mu(E)\,=\, \mu(g^*V)$ \cite[p.~23, Lemma 1.5.5]{HL}; this
subbundle $E$ is called the socle of $g^*V$. Since
$g^*V$ is preserved by the action of the Galois group $\text{Gal}(g)$ on 
$g^*V$, there is a unique subbundle $E'\, \subset\, V$ such that
$$E\,=\, g^*E'\, \subset\, g^*V\, .$$ As
$V$ is stable, we conclude that $E'\,=\, V$, and hence $g^*V$ is polystable. So we have a
direct sum decomposition
\begin{equation}\label{gv}
g^*V\,=\,\bigoplus_{j=1}^m V_j\, ,
\end{equation}
where each $V_j$ is stable with $\mu(V_j)\,=\, \mu(g^*V)$.

Take any $1\,\leq\, j\, \leq\, m$, where $m$ is the integer in \eqref{gv}.
Since $V_j$ is stable, and $\widehat{g}$ in \eqref{lcd} is Galois (this was shown in the
proof of Lemma \ref{negetiveslope}),
repeating the above argument involving the socle we conclude that $\widehat{g}^*V_j$ is also
polystable. On the other hand, as $\widehat{g}$ is genuinely ramified (see the proof of Lemma
\ref{negetiveslope}), from Lemma~\ref{lemma3.8} it follows that
\begin{equation}\label{m1}
H^0(\widehat{C},\, \text{End}(\widehat{g}^*V_j))\, =\, H^0(\widehat{D},\, \text{End}(V_j))\, .
\end{equation}
But $H^0(\widehat{D},\, \text{End}(V_j))\,=\, k$, because $V_j$ is stable. Hence from \eqref{m1}
we know that $H^0(\widehat{C},\, \text{End}(\widehat{g}^*V_j))\, =\,k$. This implies that
$\widehat{g}^*V_j$ is stable, because it is polystable.

Since $\widehat{g}^*V_j$ is stable, and $\pi_2\circ h\,=\, \widehat{g}$ (see \eqref{lcd}),
we conclude that $\pi_2^*V_j$ is also stable with
$$
\mu(\pi_2^*V_j)\,=\, \mu(\pi_2^*g^*V)
$$
for all $1\,\leq\, j\, \leq\, m$. This implies that
\begin{equation}\label{ho2}
\pi_1^*f^*V \,= \,\pi_2^*g^*V \,=\, \bigoplus_{j=1}^m \pi_2^* V_j
\end{equation}
is polystable. 

The map $\pi_2$ is genuinely ramified because $\widehat g$ is genuinely ramified
(see the proof of Lemma \ref{negetiveslope}) and ${\widehat g}\,= \,h\circ\pi_2$. Indeed, if
$\pi_2$ factors through an \'etale covering of $\widehat D$, then the
genuinely ramified map ${\widehat g}$ factors through that \'etale covering of
$\widehat D$, and hence from Proposition \ref{genuinerampi1surj}
it follows that $\pi_2$ is genuinely ramified.

Since $\pi_2$ is genuinely ramified, and each $V_j$ in \eqref{ho2} is stable,
from Lemma \ref{lemma3.8} it follows that
\begin{equation}\label{ho1}
H^0(C\times_D\widehat{D},\, \text{Hom}(\pi_2^* V_i,\, \pi_2^* V_j))\,=\,
H^0(\widehat{D},\, \text{Hom}(V_i,\, V_j))
\end{equation}
for all $1\,\leq\, i,\, j\, \leq\, m$. We know that $V_i$ and $\pi_2^* V_i$ are stable. So from
\eqref{ho1} we conclude that $V_i$ is isomorphic to $V_j$ if and only if $\pi_2^* V_i$ is
isomorphic to $\pi_2^* V_j$. From \eqref{ho1} it also follows that
\begin{equation}\label{ho3}
H^0(C\times_D\widehat{D},\, \text{End}(\pi^*_2g^*V))\,=\,
H^0(\widehat{D},\, \text{End}(g^*V))\, ;
\end{equation}
we note that this also follows from Lemma \ref{lemma3.8}.

The vector bundle $f^*V$ on $C$ is semistable, because $V$ is semistable and $f$ is separable.
Let
\begin{equation}\label{m-1}
0\, \not=\, S\,\subset\, f^*V
\end{equation}
be a stable subbundle with
\begin{equation}\label{k-1}
\mu(S)\,=\, \mu(f^*V)\, .
\end{equation}

Since $S$ is stable with $\mu(S)\,=\, \mu(f^*V)$, and the map $\pi_1$ is Galois, using the earlier 
argument involving the socle we conclude that
\begin{equation}\label{m2}
\widetilde{\mathbb S}\,:= \,\pi_1^*S\, \subset\, \pi_1^*f^*V\,=\, \pi_2^*g^*V
\,=\, \bigoplus_{j=1}^m \pi_2^* V_j\, =:\, \widetilde{V}
\end{equation}
is a polystable subbundle with $\mu(\widetilde{\mathbb S})\,=\, \mu(\widetilde{V})$.

Consider the associative algebra $H^0(C\times_D\widehat{D},\, {\rm End}(\widetilde{V}))$, where $\widetilde{V}$
is the vector bundle in \eqref{m2}. Define the right ideal
\begin{equation}\label{th}
\Theta \, :=\, \{\gamma\, \, \in\,H^0(C\times_D\widehat{D},\, {\rm End}(\widetilde{V}))\, \mid\,
\gamma(\widetilde{V})\, \subset\, \widetilde{\mathbb S}\}\, \subset\,
H^0(C\times_D\widehat{D},\, {\rm End}(\widetilde{V}))\, ,
\end{equation}
where $\widetilde{\mathbb S}$ is the subbundle in \eqref{m2}. The subbundle
$\widetilde{\mathbb S}\, \subset\, \widetilde{V}$ is a direct summand, because $\widetilde{V}$ is
polystable, and $\mu(\widetilde{\mathbb S})\,=\, \mu(\widetilde{V})$. Consequently,
$\widetilde{\mathbb S}$ coincides with the subbundle
generated by the images of endomorphisms lying in the right ideal ${\Theta}$. Since $\widetilde{V}$
is semistable, the image of any endomorphism of it is a subbundle.

Consider $\widetilde{V}$ in \eqref{m2}. The identification
$$
H^0(C\times_D\widehat{D},\, {\rm End}(\widetilde{V})) \,=\,
H^0(\widehat{D},\, \text{End}(g^*V))
$$
in \eqref{ho3} preserves the associative algebra structures of
$$H^0(C\times_D\widehat{D},\, {\rm End}(\widetilde{V}))\ \ \text{ and }\ \ 
H^0(\widehat{D},\, \text{End}(g^*V))\, ,$$
because it sends any $\gamma\, \in\, H^0(\widehat{D},\,{\rm End}(g^*V))$ to $\pi^*_2\gamma$. Let
\begin{equation}\label{th2}
\widetilde{\Theta}\,\subset\,H^0(\widehat{D},\, \text{End}(g^*V))
\end{equation}
be the right ideal that corresponds to $\Theta$ in \eqref{th} by the identification in \eqref{ho3}.
Let
\begin{equation}\label{m4}
\overline{\mathcal S}\, \subset\, g^*V
\end{equation}
be the subbundle generated by the images of endomorphisms lying in the right ideal $\widetilde{\Theta}$
in \eqref{th2}. Since $g^*V$ is semistable, the image of any endomorphism of it is a subbundle. From the
above construction of $\overline{\mathcal S}$ it follows that
$$
\widetilde{\mathbb S}\,=\, \pi^*_2 \overline{\mathcal S}\, ,
$$
where $\widetilde{\mathbb S}$ is the subbundle in \eqref{th}.

The isomorphism in \eqref{ho3} is equivariant for the actions of
the Galois group $\text{Gal}(\pi_1)\,=\, \text{Gal}(g)$ on
$$
H^0(C\times_D\widehat{D},\, {\rm End}(\widetilde{V}))\,=\,
H^0(C\times_D\widehat{D},\, {\rm End}(\pi^*_1f^*V))
$$
and $H^0(\widehat{D},\,{\rm End}(g^*V))$, because the isomorphism
sends any $\gamma\, \in\, H^0(\widehat{D},\,{\rm End}(g^*V))$ to $\pi^*_2\gamma$. Since
$\widetilde{\mathbb S}\,= \,\pi_1^*S$ in \eqref{m2} is preserved under the action of
$\text{Gal}(\pi_1)$ on $\pi^*_1f^*V$, it follows that the action of $\text{Gal}(\pi_1)$
on $H^0(C\times_D\widehat{D},\, {\rm End}(\pi^*_1f^*V))$ preserves the right ideal
$\Theta$ in \eqref{th}. These together imply that the action of $\text{Gal}(g)$ on
$H^0(\widehat{D},\,{\rm End}(g^*V))$ preserves the right ideal $\widetilde{\Theta}$ in \eqref{th2}.
Consequently, the subbundle
$$
\overline{\mathcal S}\, \subset\, g^*V
$$
in \eqref{m4} is preserved under the action of $\text{Gal}(g)$ on $g^*V$.

Since $\overline{\mathcal S}$ is preserved under the action of $\text{Gal}(g)$ on $g^*V$,
there is a unique subbundle
$$
{\mathbb S}_0\, \subset\, V
$$
such that $\overline{\mathcal S}\,=\,g^*{\mathbb S}_0\, \subset\, g^*V$. Given that $V$ is stable,
and $\mu({\mathbb S}_0)\,=\, \mu(V)$ (this follows from \eqref{k-1}), we
now conclude that ${\mathbb S}_0\, =\, V$. Hence the subbundle $S$ in \eqref{m-1} coincides with $f^*V$.
Therefore, we conclude that $f^*V$ is stable.
\end{proof}

\section{Characterizations of genuinely ramified maps}

Let $D$ be an irreducible smooth projective curve, and let
$$\phi\,:\, X\,\longrightarrow\, D$$ be a nontrivial \'etale covering with $X$
irreducible. Let $L$ be a line bundle on $X$ of degree one.

\begin{proposition}\label{prop1}\mbox{}
\begin{enumerate}
\item The direct image $\phi_*L$ is a stable vector bundle on $D$.

\item The pulled back bundle $\phi^*\phi_*L$ is not stable.
\end{enumerate}
\end{proposition}

\begin{proof}
Let $\delta$ be the degree of
the map $\phi$; note that $\delta\, >\, 1$, because $\phi$ is nontrivial.

We have $\text{deg}(\phi_*L)\,=\, \text{deg}(L)\,=\, 1$ \cite[p.~306, Ch.~IV, Ex.~2.6(a)~and~2.6(d)]{Ha}.
This implies that
$$
\text{deg}(\phi^*\phi_*L)\,=\, \delta\cdot\text{deg}(L)\,=\,\delta\, .
$$
We have a natural homomorphism
\begin{equation}\label{hw0}
H\, :\, \phi^*\phi_*L\, \longrightarrow\, L\, .
\end{equation}
This $H$ has the following property: For any coherent subsheaf $W\, \subset\, \phi_*L$, the restriction
of $H$ to $\phi^*W\, \subset\, \phi^*\phi_*L$
\begin{equation}\label{hw}
H_W\, :=\, H\vert_{\phi^*W}\, :\, \phi^* W\, \longrightarrow\, L
\end{equation}
is a nonzero homomorphism. Note that for any point $y\,\in\, D$, the fiber
$(\phi^*\phi_*L)_y$ is $H^0(\phi^{-1}(y),\, L\vert_{\phi^{-1}(y)})$, and hence
a nonzero element of $(\phi^*\phi_*L)_y$ must be nonzero at some point of $\phi^{-1}(y)$.

We will first show that $\phi_*L$ is semistable.
To prove this by contradiction, let $V\, \subset\, \phi_*L$ be a semistable subbundle with
\begin{equation}\label{hw2}
\mu(V)\, >\, \mu(\phi_*L)\,=\, \frac{1}{\delta}\, .
\end{equation}
Consider the nonzero homomorphism
$$
H_V\, :=\, H\vert_{\phi^*V}\, :\, \phi^* V\, \longrightarrow\, L
$$
in \eqref{hw}. We have $\mu(\phi^* V)\,=\, \delta\cdot \mu(V)\, >\, 1$ (see \eqref{hw2}), and
also $\phi^* V$ is semistable because $V$ is so. Consequently, $H_V$ contradicts \eqref{a1}. As
$\phi_*L$ does not contain any subbundle $V$ satisfying \eqref{hw2},
we conclude that $\phi_*L$ is semistable.

Since $\text{rk}(\phi_*L)$ is coprime to $\text{deg}(\phi_*L)$, the semistable vector
bundle $\phi_*L$ is also stable. This proves statement (1).

The vector bundle $\phi^*\phi_*L$ is not stable, because the homomorphism $H$ in \eqref{hw0} is nonzero
and $\mu(\phi^*\phi_*L)\,=\, \mu(L)$.
\end{proof}

\begin{proposition}\label{prop2}
Let $f\,:\,C\,\longrightarrow\,D$ be a nonconstant separable morphism
between irreducible smooth projective curves such that $f$ is not genuinely ramified.
Then there is a stable vector bundle $E$ on $D$ such that $f^*E$ is not stable.
\end{proposition}

\begin{proof}
Since $f$ is not genuinely ramified, from Proposition \ref{genuinerampi1surj} we know that
there is a nontrivial \'etale covering
$$\phi\,:\, X\,\longrightarrow\, D$$
and a map $\beta\, :\, C\,\longrightarrow\, X$ such that $\phi\circ\beta\,=\, f$. As in Proposition
\ref{prop1}, take a line bundle $L$ on $X$ of degree one. The vector bundle
$\phi_*L$ is stable by Proposition \ref{prop1}(1).

The vector bundle $\phi^*\phi_*L$ is not stable by Proposition \ref{prop1}(2). Therefore,
$$
f^*(\phi_*L)\, =\, \beta^*\phi^*(\phi_*L)\,=\, \beta^*(\phi^*\phi_*L)
$$
is not stable.
\end{proof}

Theorem \ref{thm1} and Proposition \ref{prop2} together give the following:

\begin{theorem}\label{thm2}
Let $f\,:\,C\,\longrightarrow\,D$ be a nonconstant separable morphism
between irreducible smooth projective curves. The map $f$ is genuinely ramified 
if and only if $f^*E$ is stable for every stable vector bundle $E$ on $D$.
\end{theorem}

\section*{Acknowledgements}

We thank both the referees for going through the paper very carefully and making numerous suggestions. The 
first-named author is partially supported by a J. C. Bose Fellowship. Both authors were supported by the
Department of Atomic Energy, Government of India, under project no.12-R\&D-TFR-5.01-0500.

\end{document}